\documentclass[12pt,twoside,a4paper]{amsart}
\usepackage{amssymb,hyperref}
\usepackage[foot]{amsaddr} 
\newcommand \Xe {\chi_\epsilon}
\newcommand \dfibar {\overline{\partial f_i}}
\newcommand \dfjbar {\overline{\partial f_j}}
\newcommand \dfJbar {\overline{\partial f_J}}
\newcommand \wbar {\overline{w}}

\newcommand \C {\mathbb{C}}
\newcommand \R {\mathbb{R}}

\newcommand \Ordo {\mathcal{O}}
\newcommand \dbar {\overline{\partial}}
\newcommand \local {\mathcal{O}_n}

\renewcommand \epsilon \varepsilon

\newtheorem{theorem}{Theorem}[section]
\newtheorem{lemma}[theorem]{Lemma}
\newtheorem{claim}[theorem]{Claim}

\newtheorem{proposition}[theorem]{Proposition}

\theoremstyle{remark}
\newtheorem{remark}[theorem]{Remark}

\numberwithin{equation}{section}

\title[]{An elementary proof of the Brian\c con-Skoda theorem}

\begin{document}

\author{Jacob Sznajdman}

\address{Mathematical Sciences, Chalmers University of Technology and G\"oteborg University\\S-412 96 G\"OTEBORG\\SWEDEN}

\email{sznajdma@chalmers.se}



\renewcommand{\abstractname}{Abstract/R\'esum\'e}
\begin{abstract}
Nous proposons une d\'emonstration \'el\'ement-aire du th\'eor\`eme de Brian\c con-Skoda.
Ce th\'eor\`eme donne un crit\`ere d'appartenance d'une fonction $\phi$ \`a un id\'eal $I$ de
l'anneau des germes de fonctions holomorphes en $0\in \C^n$; plus pr\'ecisement, l'appartenance est
\'etablie sous l'hypoth\`ese qu'une fonction d\'epend-ante de $\phi$ et $I$ soit de carr\'e localement sommable.
En partiulier, si $I$ est engendr\'e par m \'el\'ements, alors $\overline{I^{\min(m,n)}}\subset I$, o\`u 
$\overline{J}$ d\'enote la cl\^oture int\'egrale d'un id\'eal $J$.

We give an elementary proof of the Brian\c con-Skoda theorem. The theorem gives a criterion
for when a function $\phi$ belongs to an ideal $I$ of the ring of germs of analytic functions at $0\in \C^n$;
more precisely, the ideal membership is obtained if a function associated with $\phi$ and $I$ is
locally square integrable. If $I$ can be generated by $m$ elements,
it follows in particular that $\overline{I^{\min(m,n)}}\subset I$,
where $\overline{J}$ denotes the integral closure of an ideal $J$.
\end{abstract}

\maketitle

\bibliographystyle{amsalpha}

\section{Introduction}\label{historical}
Let $\local$ be the ring of germs of holomorphic functions at $0\in\C^n$.
The integral closure $\overline{I}$ of an ideal $I$ is the set of all $\phi\in\local$ such that
\begin{align}\label{closure}
	\phi^N + a_1 \phi^{N-1} + \dots + a_N = 0,
\end{align}
for some integer $N\geq 1$ and some $a_k \in I^k,\ k=1,\dots,N$.

By a simple estimate, \eqref{closure} implies that there exists a constant $C$ such that
\begin{align}\label{intolik}
	|\phi| \leq C |f|,
\end{align}
where $|f|$ is defined as $\sum |f_i|$ for any generators $f_i$ of $I$. It is easy to see that the choice of generators $f_i$
does not affect whether $\phi$ satisfies \eqref{intolik} for some $C$ or not.

Conversely, \eqref{intolik} implies that $\phi \in \overline{I}$
(however, we do not need this in the present paper), which is
a consequence of Skoda's theorem, \cite{skodal2} and a well-known determinant trick, see for example \cite{demailly}, (10.5), Ch. VIII. Another proof is given in (the republication) \cite{lejeunetessier}. 

\begin{theorem}[Brian\c con-Skoda]\label{BS}
	Let $I$ be an ideal of $\local$ generated by
	$m$ germs $f_1,\dots,f_m$. Then $\overline{I^{\min(m,n)+l-1}} \subset I^{l}$ for all integers $l\geq 1$.
\end{theorem}

As noted above, $\phi \in \overline{I^{\min(m,n)+l-1}}$ implies that $|\phi|\leq C |f|^{\min(m,n)+l-1}$.
Thus it suffices to show that any $\phi\in\local$ that satisfies this size condition belongs to $I^l$, in order to prove Theorem~\ref{BS}.

Another ideal that is common to consider is $\hat{I}^{(k)}$ which consists
of all $\phi \in \local$ such that
\begin{align}
	\int_U |\phi|^2|f|^{-2(k+\epsilon)}dV < \infty,
\end{align}
for some neighbourhood $U$ of $0\in\C^n$ and some (sufficiently small) $\epsilon > 0$, where $dV$ is the Lebesgue measure.

Lemma~\ref{weierstrass} implies that $\overline{I^k} \subset \hat{I}^{(k)}$.
The following theorem is thus a stronger version of Theorem~\ref{BS}:

\begin{theorem}\label{BSl2}
	For an ideal $I$ as in Theorem \ref{BS}, we have
	\begin{align*}
		\hat{I}^{(\min(m,n)+l-1)} \subset I^{l},
	\end{align*}
	for all integers $l \geq 1$.
\end{theorem}

In 1974 Brian\c con and Skoda, \cite{brianconskoda}, showed Theorem \ref{BSl2} as an immediate consequence of
Skoda's $L^2$-division-theorem, \cite{skodal2}. Usually Theorem \ref{BS} is the one referred to as the Brian\c con-Skoda
theorem.

An algebraic proof of Theorem~\ref{BS} was given by Lipman and Tessier in \cite{lipmantessier}. Their paper also contains a historical summary.
An account of more recent developments and an elementary algebraic proof of the result is found in 
Schoutens \cite{schoutensbs03}. 

Berenstein, Gay, Vidras and Yger~\cite{bgvy} proved Theorem~\ref{BS} for $l=1$ by finding a representation
$\phi = \sum u_i f_i$ with $u_i$ as explicit integrals. However, some of their estimates rely on Hironaka's theorem
on resolutions of singularities.

In this paper, we provide a completely elementary proof along these lines. The key point is an $L^1$-estimate
(Proposition \ref{stortlemma}), which will be used in Section~\ref{elementary}.

\textbf{Acknowledgements} I am greatful to Mats Andersson for introducing me to the subject
and providing many helpful comments and ideas. I also want to thank the referee who read
the paper very carefully and gave many valuable suggestions.

\section{The Main Estimate}\label{local}

In order to state Proposition \ref{stortlemma}, we will first recall the notion of the (standard) norm
of a differential form in $\C^n$. If $x_i$ and $y_i$, $1\leq i \leq n$, are standard coordinates for $\C^n=\R^{2n}$,
this norm is uniquely determined by demanding that the forms $dx_{i_1}\wedge \dots \wedge dx_{i_j} \wedge dy_{i_{j+1}} \wedge \dots \wedge dy_{i_k}$ constitute an orthonormal basis (over $\C$) of $\bigwedge^k T_p^*\C^n$.

\begin{proposition}\label{stortlemma}
	Let $f_1, f_2, \dots, f_m$ be generators of an ideal $I \subset \local$, and assume that
	$\phi\in \hat{I}^{(k)}$.
	Then for any integer $1\leq r \leq m$,
	\begin{equation*}
		\frac{|\phi|\cdot \left|\partial f_1 \wedge\dots\wedge \partial f_r\right|}{|f|^{k+r}}
	\end{equation*}
	is locally integrable at the origin.
\end{proposition}

\begin{remark}
Using a Hironaka resolution,
the proof of Proposition~\ref{stortlemma}
can be reduced to the case when every $f_i$ is a monomial,
and then the proof becomes much easier.
We proceed however with elementary arguments.
\end{remark}

\begin{lemma}\label{weierstrass}
	For any ideal $I = (f_1 , \dots , f_m) \neq (0)$, there is a positive number $\delta$ such that
	$1/|f|^\delta$ is locally integrable at the origin.
	\begin{proof}[Proof]
		By considering $F=f_1 \cdot f_2 \cdot \ldots \cdot f_m$
		(remove any $f_j$ that are identically zero), it suffices to show that
		$1/|F|^\delta$ is locally integrable.
		We can assume that $F$ is a Weierstrass polynomial and we consider the integral of
		$1/|F|^\delta$ on $\Omega=D \times \Delta$, where $D$ is a disk and $\Delta = D^{n-1}$.
		By choosing $D$ small enough,
		Rouch\'e's theorem gives that $F$ has the same number of roots, $s$, on each slice
		$S_p=D\times\{p\}$, $p\in \Delta$.
		We partition $S_p$ into sets $E^p_j$, one for each root $\alpha_j(p)\in S_p$,
		such that $E^p_j$ consists of those points which are closer to $\alpha_j(p)$ than to the other roots.
		We have $F(z,p)=\prod_1^s (z-\alpha_j(p))$, so on $E^p_j$ we get
		$1/|F|^\delta \leq |z-\alpha_j(p)|^{-\delta s}$.
		If $\delta$ is sufficiently small, we thus get a uniform bound for the (one variable)
		integral of $1/|F|^\delta$ on $S_p$. Fubini's theorem then gives the integrability on $\Omega$.
	\end{proof}
\end{lemma}

\begin{proof}[Proof of Proposition \ref{stortlemma}]
	We assume for the sake of simplicity that $r=m$, but the proof works for the other cases as well. 
	We begin by applying H\"older's inequality to the product of
	${|\phi|}/{|f|^{k+\delta'/2}}$ and
	${\left|\partial f_1 \wedge\dots\wedge \partial f_m\right|}/{|f|^{m-\delta'/2}}$. Assume that $\delta'$ is
	small enough to make the first factor $L^2$-integrable.
	It thus suffices to show that
	\begin{align*}   
		F=\frac{\left|\partial f_1\wedge\dots\wedge\partial f_m\right|^2}
			{\prod_1^m|f_j|^{2-\delta}}
	\end{align*}
	is locally integrable for any $\delta>0$.
	We will proceed to show that this is a consequence of
	the Chern-Levine-Nirenberg inequalities.
	The special case of these inequalities that is needed here will be proved
	without explicitly relying on facts about positive forms or plurisubharmonic functions.
	For a shorter proof of the Chern-Levine-Nirenberg inequalities, which involves these notions,
	see \cite{demailly} (3.3), Ch. III.

	Let us first set
	\begin{align*}
		\beta = \frac i2\partial \dbar |\zeta|^2 =
			\frac i2\sum d\zeta_j \wedge d\overline{\zeta_j}, \quad \text{and} \quad
		\beta_k = \frac{\beta^k}{k!}.
	\end{align*}
	Then $\beta_n$ is the Lebesgue measure $dV$. A simple argument gives that for any $(1,0)$-forms $\alpha_j$,
	\begin{align}\label{metricform}
		\frac i2\alpha_1 \wedge \overline{\alpha_1} \wedge \dots \wedge
		 \frac i2\alpha_p \wedge \overline{\alpha_p} \wedge \beta_{n-p} = |\alpha_1 \wedge\dots\wedge\alpha_p|^2dV.
	\end{align}
	Fix a sufficiently small $\delta > 0$ as in Lemma~\ref{weierstrass}. We will need at least $\delta < 2$ in the
	sequel.
	We now compute
	\begin{align*}
		\partial\dbar{(|f_j|^2+\varepsilon)}^{\delta/2} 
		=\frac{\delta}{2}
			\left(1+\frac{\left(\frac{\delta}{2}-1\right)|f_j|^2}{|f_j|^2+\varepsilon}\right)
			{(|f_j|^2+\varepsilon)}^{\delta/2-1}\partial f_j\wedge\dfjbar,
	\end{align*}
	which yields that
	\begin{align}\label{claim}
		\frac{i\partial f_j\wedge \dfjbar}{{\left(|f_j|^2+\varepsilon\right)}^{1-\delta/2}} =
			G_j i\partial\dbar{(|f_j|^2 + \varepsilon)}^{\delta/2},
	\end{align}
	where
	\begin{align*}
		G_j = \frac{2}{\delta}\left[1+\left(\frac{\delta}{2}-1\right)
				\frac{|f_j|^2}{|f_j|^2+\varepsilon}\right]^{-1}.
	\end{align*}
	Observe that
	\begin{align}\label{Gest}
		\left(\frac 2\delta\right) \leq G_j \leq \left(\frac 2\delta\right)^2.
	\end{align}
	We introduce forms $F^k_\varepsilon dV$	by setting
	\begin{align}\notag 
		F^k_\varepsilon dV &= \frac{\left|\partial f_k \wedge \dots \wedge \partial f_m\right|^2}
			{\prod_k^m\left(|f_j|^2+\varepsilon\right)^{1-\delta/2}}dV
		= \frac{\prod_k^m\left(\frac i2\partial f_j \wedge \dfjbar\right) \wedge \beta_{n+k-m-1}}
				{\prod_k^m\left(|f_j|^2+\varepsilon\right)^{1-\delta/2}}=\notag\\ \label{ddfBeta}
		&= \prod_k^m G_j \frac i2\partial\dbar \left(|f_j|^2+\varepsilon\right)^{\delta/2} \wedge \beta_{n+k-m-1}.
	\end{align}
	Note that $F^1_\varepsilon dV$ is a regularization of $FdV$.
	From the equality $|w \wedge \wbar|=2^p|w|^2$, that holds for all $(p,0)$-forms $w$, and \eqref{claim}, we get
	\begin{align}\label{ddfdV}
		F^k_\varepsilon dV = \frac{\left|\prod_k^m\left(\frac i2\partial f_j \wedge \dfjbar\right)\right| dV}
			{\prod_k^m\left(|f_j|^2+\varepsilon\right)^{1-\delta/2}} =
			\left|\prod_k^m G_j \frac i2\partial\dbar \left(|f_j|^2+\varepsilon\right)^{\delta/2}\right| dV.
	\end{align}
	Comparing \eqref{ddfBeta} with \eqref{ddfdV}, we get
	\begin{align}\label{withoutG}
		H^k_\varepsilon dV :=
		\prod_k^m i\partial\dbar \left(|f_j|^2+\varepsilon\right)^{\delta/2} \wedge \beta_{n+k-m-1} =
			\left|\prod_k^m i\partial\dbar \left(|f_j|^2+\varepsilon\right)^{\delta/2}\right| dV.
	\end{align}

	Let $B$ be a ball about the origin and let $\chi_B$ be a smooth cut-off function supported in a
	concentric ball of twice the radius. We now use \eqref{ddfdV}, \eqref{withoutG} and \eqref{Gest}
	and integrate by parts (going from the second to the third line below) to see that
	\begin{align*}
		& \int_B F^1_\varepsilon dV \leq C_\delta \int \chi_B \left|i\partial\dbar
			\left(|f_1|^2+\varepsilon\right)^{\frac{\delta}{2}} \wedge\dots\wedge
			i\partial\dbar \left(|f_m|^2+\varepsilon\right)^{\frac{\delta}{2}}\right|dV  = \\
		& =\ C_\delta \int \chi_B i\partial\dbar \left(|f_1|^2+\varepsilon\right)^{\delta/2} \wedge\dots\wedge
			i\partial\dbar \left(|f_m|^2+\varepsilon\right)^{\delta/2}\wedge \beta_{n-m}  = \\
		& =\ C_\delta \left|\int (\partial\dbar\chi_B) \left(|f_1|^2+\varepsilon\right)^{\delta/2} \wedge\dots\wedge
			i\partial\dbar \left(|f_m|^2+\varepsilon\right)^{\delta/2}\wedge \beta_{n-m}\right|  \leq \\
		& \leq\  C_1 C_\delta \sup_{2B}|f_1|^\delta \int_{2B} \left|i\partial\dbar\left
			(|f_2|^2+\varepsilon\right)^{\frac{\delta}{2}}\wedge\dots\wedge i\partial\dbar
			\left(|f_m|^2+\varepsilon\right)^{\frac{\delta}{2}}\right|dV  \leq \\
		& \leq\ C_1 C_\delta \sup_{2B}|f_1|^\delta \int \chi_{2B} H^2_\varepsilon dV, 
 	\end{align*}
	where $C_\delta = 2^m/\delta^{2m}$ and $C_1 = \sup \chi_B$. Should the reader have any doubts about the
	integration by parts, note that
	$d(\alpha \wedge \beta \wedge \gamma) =
		\partial \alpha \wedge \beta \wedge \gamma + \alpha \wedge \partial \beta \wedge \gamma$,
	for any function $\alpha$ and forms $\beta$ and $\gamma$ such that $\gamma$
	is a closed $(n-1,n-1)$-form and $\beta$ is a $(0,1)$-form. A similar relation holds for the $\dbar$-operator.
	Since the second integral on the first line in the calculation above is nothing but $\int \chi_B H^1_\varepsilon dV$,
	we can proceed by induction over $k$ to obtain
	\begin{align*}
		\int_B |F_\varepsilon|dV \leq  \frac {C}{\delta^{2m}}\sup_{2^{m+1}B}|f_1 \cdot \ldots \cdot f_m|^\delta
		<\infty,
	\end{align*}
	so if we let $\varepsilon$ tend to zero, we get the desired bound.
\end{proof}
\begin{remark}
	It is not hard to see that essentially the same proof gives that
	${\left|\partial f_1 \wedge\dots\wedge \partial f_r\right|}/{\prod_1^r|f_i|}$
	is locally integrable.
\end{remark}

\section{Division by weighted integral formulas}\label{weights}
We will use a division formula introduced in \cite{bobformula},
but for convenience, we use the formalism from \cite{matsaintrep1} to describe it.

Consider a fixed point $z\in \C^n$ and define the operator\\ $\nabla_{\zeta -z} = \delta_{\zeta - z} - \bar \partial$,
where $\delta_{\zeta-z}$ is contraction with the vector field 
\begin{align*}
	2\pi i \sum_1^n {\left(\zeta_k - z_k\right)\frac{\partial}{\partial \zeta_k}}.
\end{align*}
Recall that $\delta_{\zeta-z}$ anti-commutes with $\dbar$.
We allow these operators to act on forms of all bidegrees. In particular, the contraction of a function
is zero.

A {\em weight} with respect to $z$ is a smooth differential form
$g = g_{0,0} + g_{1,1} + \dots + g_{n,n}$ such that
$\nabla_{\zeta -z}g=0$ and $g_{0,0}(z) = 1$. The subscripts denote bidegree.

Let $s$ be any $(1,0)$-form such that $\delta_{\zeta -z} s = 1$ outside of $\{\zeta=z\}$, e.g.,
\begin{align*}
	s=\frac{\partial |\zeta|^2}{2\pi i \left(|\zeta|^2 - \overline{\zeta} \cdot z \right)},
\end{align*}
where the dot sign denotes the pairing given by $a \cdot b = \sum a_i b_i$. Next we set
\begin{align*}
	u =s + s\wedge \dbar s + \dots + s\wedge (\dbar s)^{n-1},
\end{align*}
which is defined whenever $s$ is defined.
We note that $\delta_{\zeta-z} \dbar s = - \dbar \delta_{\zeta-z} s \\= -\dbar 1 = 0$.
Since $s\wedge (\dbar s)^n$ must vanish, we have $(\dbar s)^n = \delta_{\zeta-z} (s\wedge (\dbar s)^n)$ = 0.
The reader may check that $\nabla_{\zeta - z} u = 1$. In fact, this can be seen
elegantly by using functional calculus of differential forms; then $u = s/{\nabla_{\zeta - z} s} =
s/{(1-\dbar s)} = s\wedge \sum_1^{n-1} (\dbar s)^k$,
and $\nabla_{\zeta-z} u = {\nabla s}/{\nabla s} = 1$.

One can construct a weight $g_z(\zeta)$ with respect to $z$, compactly supported in the ball of radius
$r+\epsilon$, such that $(z,\zeta) \mapsto g_z(\zeta)$ is holomorphic in $z$ in the ball of radius
$r-\epsilon$.
This is accomplished by setting
\begin{align*}
	g_z(\zeta) = \chi - \dbar \chi \wedge u,
\end{align*}
where $\chi$ is a cut-off
function that is $1$ whenever $|\zeta| \leq r-\epsilon$ and $0$ whenever $|\zeta| > r + \epsilon$. Note
that $u$ is well-defined on the support of $\dbar \chi$.
We see that $g_z$ is a weight since $\nabla_{\zeta - z}$ is an anti-derivation;
$\nabla_{\zeta - z} g_z = -\dbar \chi  + \dbar \delta_{\zeta - z} \chi \wedge u + \dbar \chi  = 0$ (as $\chi$ is a function,
we have $\delta_{\zeta - z} \chi = 0$).

\begin{proposition}
If $g$ is a weight with respect to $z$ which has compact support, and if $\phi$ is holomorphic
in a neighbourhood of the support of $g$, then
	\begin{align}\label{weightformula}
		\phi(z) = \int \phi(\zeta) g(\zeta).
	\end{align}
\begin{proof}
	As in the construction of a weight with compact support above, we define forms
	\begin{align*}
		b=\frac{\partial |\zeta-z|^2}{2\pi i|\zeta-z|^2}
	\end{align*}
	and $u=b \wedge \sum\left(\dbar b\right)^k$ such that $\delta_{\zeta -z}b = 1$ and $\nabla_{\zeta - z}u=1$
	hold outside of $\{\zeta = z\}$. The highest degree term of $u$ is the Bochner-Martinelli kernel.
	We now want to determine the residue
	$R=1-\nabla_{\zeta -z}u$ (where $\nabla_{\zeta-z}$ is taken in the sense of currents)
	at $\{\zeta = z\}$. The $(k,k-1)$ bidegree component
	$u_{k,k-1}$ of $u$ is $\Ordo(|\zeta-z|^{-2k+1})$, so only the highest component,
	$\dbar u_{n,n-1}=\dbar (b\wedge (\dbar b)^{n-1})$ of $\nabla_{\zeta -z} u$ will contribute to the residue.
	Using Stokes' theorem, it is easy to check that $R=[z]$, the point evaluation current at $z$. 
	Clearly $\nabla_{\zeta-z}(\phi g) = 0$, so $\nabla_{\zeta-z} (u\wedge \phi g) = \phi g - [z]\wedge \phi g$.
	Taking highest order terms, we get
	\begin{align*}
		d(u\wedge \phi g)_{n,n-1} = \dbar (u\wedge \phi g)_{n,n-1} = [z]\wedge \phi g_{0,0} - \phi g_{n,n} 
		=[z]\wedge \phi - \phi g_{n,n},
	\end{align*}
	so by Stokes's theorem
	\begin{align*}
		\int \phi(\zeta) g(\zeta) = \int \phi(\zeta) g_{n,n}(\zeta) = [z].\phi = \phi(z).
	\end{align*}
\end{proof}
\end{proposition}

\section{Finishing the proof of Theorem 1.2}\label{elementary}

We now begin constructing a weight associated with Berndtsson's division formula for an ideal $I\subset \local$.
Take $h=(h_i)$ to be an $m$-tuple of so called Hefer forms with respect to the generators $f_i$ of $I$;
these (germs of) $(1,0)$-forms are holomorphic in $2n$ variables, and satisfy
$\delta_{\zeta - z}h_i = f_i(\zeta) - f_i(z)$. To see that $h$ exists, write
\begin{align*}
	f_i(\zeta) - f_i(z) = \int_0^1 \frac{d}{dt} f_i(z+t(\zeta-z)) dt,
\end{align*}
and compute the derivative inside the integral.
Define $\sigma_i = \bar f_i/|f|^2$ and 
let $\chi_\varepsilon = \chi(|f|/\varepsilon)$ be a smooth cut-off function, where $\chi$ is approximatively
the characteristic function for $[1,\infty)$. Recall that the dot sign refers to the pairing $a\cdot b = \sum a_i b_i$.
We now set
\begin{align*}
	\mu=\min(m,n+1)
\end{align*}
and define the weight
\begin{align}
	\begin{split} \label{decomp}
		g_B&=\left(1-\nabla_{\zeta - z}
			\left(h \cdot \chi_\varepsilon \sigma\right)\right)^\mu \\
		& = {\left(1-\Xe + f(z) \cdot \Xe \sigma
			+ h \cdot \dbar\left(\Xe \sigma \right)\right)}^\mu = \\
		&= f(z) \cdot A_\varepsilon + B_\varepsilon,
	\end{split}
\end{align}
where
\begin{align}
	\label{Apart}
	A_\varepsilon& = \sum_{k = 0}^{\mu - 1} C_k
		\Xe \sigma {\left[f(z) \cdot \Xe \sigma\right]}^k
			{\left[1-\Xe + h \cdot \dbar\left(\Xe \sigma \right)\right]}^{\mu -k -1}
	\intertext{and}
	\label{Bpart}
	B_\varepsilon & =  {\left(1-\Xe+h\cdot\dbar\left(\Xe\sigma\right)\right)}^\mu.
\end{align}

For convenience, we assume that $l=0$ in Theorem \ref{BSl2}. The proof goes through verbatim for general $l$ by
just replacing $\mu$ with $\mu + l$ in the definition of $g_B$.

Let $g$ be any weight with respect to $z$ which has compact support and is holomorphic in $z$ near $0$.
Substitution of the last line of \eqref{decomp} into \eqref{weightformula} applied to the weight $g_B\wedge g$ yields
\begin{align}\label{AandBpart}
	\phi(z) = f(z) \cdot \int{\phi(\zeta)A_\varepsilon \wedge g} + \int{\phi(\zeta) B_\varepsilon \wedge g}.
\end{align}
To obtain the division we will show two claims:
\begin{claim}\label{residydod}
	The second term in \eqref{AandBpart},
	\begin{align*}
		\int{\phi(\zeta) B_\varepsilon \wedge g},
	\end{align*}
	converges uniformly to zero for small $|z|$.
\end{claim}
\begin{claim}\label{intconv}
	If $m \leq n$, the tuple of integrals in \eqref{AandBpart},
	\begin{align*}
		\int{\phi(\zeta)A_\varepsilon \wedge g},
	\end{align*}
	converges uniformly as $\varepsilon \to 0$.
\end{claim}
We give an argument for the case $m > n$ of Theorem~\ref{BSl2} at the end of the paper.
Letting $\varepsilon$ go to zero in \eqref{AandBpart}, these claims
give that $\phi \in I$.

To prove Claim \ref{residydod}, we will soon find a function $F(\zeta)$ integrable near $\zeta=0$,
such that $|\phi(\zeta) B_\varepsilon|\leq F$.
Now we note that the integrand of Claim~\ref{residydod} has support on the set $S_\epsilon = \{|f|\leq 2\epsilon\}$;
outside of $S_\varepsilon$, we have that $\Xe =1$, so
$B_\epsilon = \left(h\cdot \dbar \sigma\right)^{\mu}$, which vanishes
regardless of whether $\mu=n+1$ or $\mu = m$. In the latter case apply $\dbar$ to
$f\cdot \sigma = 1$ to see that $\dbar \sigma$ is linearly dependent. 
Thus for small $|z|$, we get
\begin{align*}
	\lim_{\varepsilon \to 0}\left|\int{\phi(\zeta) B_\varepsilon \wedge g}\right|
		\leq C \lim_{\varepsilon \to 0}\int_{S_\epsilon} F = 0,
\end{align*}
where we used that $g$ is smooth.

The existence of $F$ is a consequence of the main estimate of the previous chapter
and a little bookkeeping that we will now carry out.
Straightforward calculations, based on the fact that $\chi'$ is bounded, give that
\begin{align}
	\label{dbarZe}
	\dbar \Xe = \Ordo(1)|f|^{-1}\sum\dfjbar \quad \text{and} \quad
	\dbar \sigma_i = \Ordo(1)|f|^{-2}\sum\dfjbar,
\end{align}
since $|f|\sim \varepsilon$ on the support of $\dbar \Xe$. Note also that $|\sigma|=|f|^{-1}$.
It is easy to see that $\Ordo(1)$ actually represents a function that does not depend on $\epsilon$.

Using these facts, as we binomially expand \eqref{Bpart}, we get that $\phi(\zeta) B_\varepsilon$ is a linear combination of
terms that are given by
\begin{align}\label{terms}
	&\phi(\zeta){\left(\dbar \Xe h \cdot \sigma\right)}^a \wedge
	{\left(\Xe h \cdot \dbar \sigma\right)}^b{\left(1 - \Xe\right)}^c =\notag\\&=
	\phi(\zeta) |f|^{-2(a+b)} \dfJbar \wedge \Ordo(1),	
\end{align}
where $a+b+c=\mu$, $J \subset \{1,2 \dots m\}$, $|J|=a+b$ and
$\dfJbar = \underset{i \in J}{\bigwedge} \dfibar$. Since $\dfJbar=0$ whenever $a+b>n$ we can assume that
$a+b \leq \min(m,n)$. We now set $F$ to be the sum of the right hand side of
\eqref{terms} over all possible $J$, i.e.
\begin{align}\label{F}
	F=\sum_{|J|\leq \min(m,n)} \phi(\zeta)|f|^{-2|J|} \dfJbar \wedge \Ordo(1).
\end{align}
Clearly $|\phi(\zeta) B_\varepsilon|\leq F$. Applying Proposition~\ref{stortlemma} with $k=\min(m,n)$ to
\eqref{F}, it follows that $F$ is indeed locally integrable. \qed\\

Before dealing with Claim \ref{intconv}, we note that there is a way around it;
clearly, the integrals in the claim are holomorphic for each $\varepsilon >0$, so the first term
in \eqref{AandBpart} belongs to $I$ for fixed $\epsilon>0$.
Thus, due to Claim~\ref{residydod}, $\phi$ is in the closure of $I$ with respect to uniform convergence.
All ideals are however closed under uniform convergence, see \cite{hormander} Chapter 6, so $\phi$ belongs to $I$.

The proof of Claim \ref{intconv} is similar to the proof of Claim~\ref{residydod}.
Since we have assumed $m\leq n$, we have $\mu=\min(m,n+1)=m$.
Expanding $\phi(\zeta)A_\varepsilon$, displayed in \eqref{Apart}, we get a linear combination of terms that are given by
\begin{align*}
	&\phi(\zeta) \sigma {\left(f(z) \cdot \Xe \sigma\right)}^k
	{\left(\dbar\Xe h \cdot \sigma \right)}^a \wedge {\left(h \cdot \dbar \sigma \right)}^b
	=\\&= \phi(\zeta)|f|^{-(1+k+2a+2b)}\dfJbar\wedge \Ordo(1), 
\end{align*}
where $a+b \leq \mu - k -1$, $k \leq \mu - 1$ and $|J| = a+b$.
The sum $1+k+2a+2b$ is at most $2\mu - 1$, and this happends when $k=0$ and $a+b=\mu - 1$.
By an argument almost identical to the one proving that
$F$ was integrable, we get an integrable upper bound for $\phi A_\epsilon$ independent of $z$ and $\epsilon$.
This is, of course, an upper bound also for the limit
\begin{align*}
	A:=\lim_{\varepsilon \to 0}A_\varepsilon =
		\sum_{k=0}^{\mu-1}
		C_k\sigma{\left[f(z)\cdot\sigma\right]}^k{\left[h\cdot\dbar\sigma\right]}^{\mu-k-1}.
\end{align*}
As in the beginning of the proof of Claim \ref{residydod},
one sees that $\int{\phi(\zeta)A_\varepsilon \wedge g}$
converges uniformly to $\int{\phi(\zeta)A \wedge g}$. \qed\\

The case $m > n$ presents an additional difficulty as our upper bound fails to be integrable. Also, $\phi A\wedge g$ will not be integrable. A remedy is to consider a reduction of the ideal $I$,
that is, an ideal $\mathfrak{a} \subset I$ generated by $n$ germs such that $\overline{\mathfrak{a}} = \overline{I}$,
see for example Lemma 10.3, Ch. VIII in \cite{demailly}. If $a_i$ generate $\mathfrak{a}$ we have that $|a| \sim |f|$,
so $\hat{\mathfrak{a}}^{(k)} = \hat{I}^{(k)}$ for any integer $k \geq 1$. Thus we have reduced to the case
$m \leq n$, which has already been proved.
\qed
\providecommand{\bysame}{\leavevmode\hbox to3em{\hrulefill}\thinspace}
\providecommand{\MR}{\relax\ifhmode\unskip\space\fi MR }
\providecommand{\MRhref}[2]{%
  \href{http://www.ams.org/mathscinet-getitem?mr=#1}{#2}
}
\providecommand{\href}[2]{#2}

\end{document}